\documentclass{article}
\usepackage{graphicx} 

\title{Global smooth solutions in a one-dimensional thermoviscoelastic model with temperature-dependent paramaters}
\author{Felix Meyer\footnote{felix.meyer@math.uni-paderborn.de}\\
{\small Universit\"at Paderborn, Institut f\"ur Mathematik}\\
{\small 33098 Paderborn, Germany} }
\usepackage{amsfonts}
\usepackage{color}
\usepackage{amsmath,amssymb,comment}

\usepackage[english]{babel}
\newtheorem{theo}{Theorem}[section]
\newtheorem{lem}[theo]{Lemma}

\newtheorem{prop}[theo]{Proposition}

\newcommand{\proof}{{\sc Proof.} \quad}
\newcommand{\proofc}{{\sc Proof} \ }
\newcommand{\be}{\begin{equation} \label}
\newcommand{\ee}{\end{equation}}
\newcommand{\bea}{\begin{eqnarray}\label}
\newcommand{\eea}{\end{eqnarray}}
\newcommand{\bas}{\begin{eqnarray*}}
\newcommand{\eas}{\end{eqnarray*}}
\newcommand{\bit}{\begin{itemize}}
\newcommand{\eit}{\end{itemize}}
\newcommand{\qed}{\hfill$\Box$ \vskip.2cm}
\newcommand{\nn}{\nonumber}
\newcommand{\R}{\mathbb{R}}

\newcommand{\pO}{\partial\Omega}

\newcommand{\eps}{\varepsilon}

\newcommand{\io}{\int_\Omega}
\newcommand{\na}{\nabla}

\newcommand{\al}{\alpha}
\newcommand{\lam}{\lambda}

\newcommand{\bom}{\overline{\Omega}}
\newcommand{\Om}{\Omega}

\newcommand{\hs}{\hspace*}

\newcommand{\vp}{\varphi}
\newcommand{\lbal}{\left\{ \begin{array}{l}}
\newcommand{\lball}{\left\{ \begin{array}{ll}}
\newcommand{\ear}{\end{array} \right.}

\newcommand{\abs}{\\[5pt]}

\renewcommand{\div}{{\rm div} \,}

\newcommand{\tm}{T_{max}}

\numberwithin{equation}{section}

\usepackage[
  style=numeric,        
  sorting=nyt,          
  doi=true,             
  giveninits=true,
  url=false,            
  isbn=false,           
  eprint=false,
  backend=biber
]{biblatex}

\AtEveryBibitem{\clearfield{month}}

\DeclareBibliographyDriver{book}{%
  \usebibmacro{bibindex}%
  \usebibmacro{begentry}%
  \usebibmacro{author/editor+others/translator+others}%
  \setunit{\labelnamepunct}\newblock
  \usebibmacro{title}%
  \newunit\newblock
  \usebibmacro{byeditor+others}%
  \newunit\newblock
  \printfield{edition}%
  \newunit\newblock
  \printlist{publisher}%
  \newunit\newblock
  \printfield{year}%
  \newunit\newblock
  \iffieldundef{url}{}{\printfield{url}\newunit\newblock}%
  \usebibmacro{finentry}%
}

\AtEveryBibitem{\clearfield{booktitle}}
\renewbibmacro{in:}{}

\DeclareFieldFormat[article,inproceedings,incollection,book]{title}{\mkbibemph{#1}}
\DeclareFieldFormat[article,inproceedings,incollection,book]{titleaddon}{\mkbibemph{#1}}

\DeclareFieldFormat{citetitle}{\mkbibemph{#1}}
\DeclareFieldFormat{booktitle}{\mkbibemph{#1}}

\DeclareFieldFormat{title}{\mkbibemph{#1}}
\renewbibmacro*{title}{
  \ifboolexpr{
    test {\iffieldundef{title}}
  }
    {}
    {\printtext{\printfield{title}}%
     \setunit{\addspace}}}

\DeclareFieldFormat[online]{title}{\mkbibemph{#1}}
\DeclareFieldFormat[unpublished]{title}{\mkbibemph{#1}}
\DeclareFieldFormat[misc]{title}{\mkbibemph{#1}}

\DeclareFieldFormat{journaltitle}{#1}
\DeclareFieldFormat{journalsubtitle}{#1}

\AtEveryBibitem{%
  \clearfield{shortjournal}%
  \clearlist{shortjournal}%
}

\DeclareFieldFormat[article]{volume}{#1}
\DeclareFieldFormat[article]{number}{\mkbibparens{#1}}

\renewbibmacro*{volume+number+eid}{
  \printfield{volume}%
  \setunit{}%
  \printfield{number}%
  \setunit{\addcomma\space}%
  \printfield{eid}}

\DeclareFieldFormat{pages}{#1}
\DeclareFieldFormat{pagetotal}{#1}
\DeclareFieldFormat{page}{#1}

\usepackage{doi}
\usepackage{hyperref}

\renewbibmacro*{doi+eprint+url}{%
  \iffieldundef{doi}
    {\iffieldundef{url}
       {}
       {\printfield{url}}%
    }
    {\printfield{doi}}%
}

\DeclareNameAlias{default}{family-given}

\newcommand{\dt}{\frac{d}{dt}}
\begin{document}

\maketitle

\begin{abstract}
\noindent 
 This manuscript is concerned with the system
\begin{align*}
	\left\{ \begin{array}{l}
	u_{tt} = (\gamma(\Theta) u_{xt})_x +  (a(x,t) u_x)_x +(f(\Theta))_x, \\[1mm]
	\Theta_t = D\Theta_{xx} + \gamma(\Theta) u_{xt}^2 + f(\Theta) u_{xt}, 
	\end{array}
    \right.
\end{align*}
which is used to describe thermoviscoelastic developments in one-dimensional Kelvin-Voigt materials. \abs 
It is assumed that $a,\gamma$ and $f$ are sufficiently smooth functions that satisfy 
$$c_\gamma<\gamma(\zeta)<C_\gamma, \quad \gamma''(\zeta) \le 0,\quad f(0)=0, \quad |f'(\zeta)|\le C_f \quad \mbox{ and } |f(\zeta)|\le C_f(1+\zeta)^\alpha \quad \mbox{ for all }\zeta\ge 0 $$
and some positive constants $c_\gamma,C_\gamma,C_f>0$ and $\alpha \in (0,5/6)$. Under these conditions, this study then establishes a result on the existence of global classical solutions for
sufficiently smooth but arbitrarily large initial data.

\end{abstract}

\newpage

\section{Introduction}

We consider the one dimensional evolution system 
\bea{HS}
	\lbal
	u_{tt} = (\gamma(\Theta) u_{xt})_x +  (a(x,t) u_x)_x +(f(\Theta))_x, \\[1mm]
	\Theta_t = D\Theta_{xx} + \gamma(\Theta) u_{xt}^2 + f(\Theta) u_{xt}, 
	\ear
\eea
which can be derived from the more general systemclass for thermo visco elastic models
\be{IntroductionSys1}
	\lball
	u_{tt} = (\gamma(\Theta,u_x) u_{xt})_x + (a(\Theta,u_x) u_x)_x - (f(\Theta,u_x))_x,  \\[1mm]
	\Theta_t = D\Theta_{xx} + \gamma(\Theta,u_x) u_{xt}^2-  f(\Theta,u_x)u_{xt}.
	\ear
\ee
Models of this type have two main physical use cases, to model the piezoelectric effect in Kelvin-Voigt type materials on the one handside and  heat generation by acoustic waves in Kelvin-Voigt-type materials on the other (see e.g. \cite{Roubicek2009} or \cite{Roubicek2013}). In consideration of the first, Piezo ceramics are excited by external electrodes using electrical voltage, to set them into vibration. This process is key in many industrial applications, including nanopositioning, ultrasound generation, and medical technology.
 Not only is the heat generated in this process an indication of inefficiency and the thereby associated energy losses, it also poses a risk to the piezoelectric components used. These are known to be very susceptible to such damage caused by high temperatures, since the polarization of the piezo ceramics is essential for the piezoelectric effect, and this effect is lost when the Curie temperature is exceeded \cite{Rupitsch2019}. 
To model the heat generation caused by acoustic wave propagation, the mechanical strain $S$ is coupled to the mechanical displacement $u$, the electric displacement field $D$ and the strength of the  electric field $E$ along the corresponding piezoelectric coupling tensor $e$ and the permittivity matrix $e$ by 
\be{IntroductionDisplacement}
 D = \eps E + eS,\quad\mbox{ with}\quad \div D=0, \qquad \mbox{ and } \qquad S=\nabla^s u,
\ee
where $\nabla^s u=\frac 12(\na u+(\na u)^T)$ represent the symmetric gradient, respectively.
When deployed in the underlying Kelvin-Voigt model of arbitrary spatial dimensions and in consideration of heat generation, in Form of thee temperature denoted as $\Theta$, the resulting wave equation formulates to 
\be{IntroductionCoupling}
	\rho u_{tt} = \div (dS_t+CS-\Theta CB)-\div(e^TE),
\ee
in which $\rho$ represents the material density, $d$ the viscosity tensor, $C$ the elastic parameter tensor and $B$ the thermal dilation tensor. 
In the one-dimensional case, for the thermal equation, the conversion of mechanical work into thermal energy can be noted in view of mechanical strain $S$ and thus displacement $u$ by
\be{IntroductionHeatGen}
Q= d S_t^2 = d u_{xt}^2,
\ee
where the unit of the quantity $Q\in \R$ is that of the density of an energy (\cite{Boley1960}). Thus the thermal equation formulates to
\be{IntroductionThermalEq}
\Theta_t = D\Theta_{xx} + du_{xt}^2 -\Theta CBu_{xt}.
\ee
If we assume that the permittivity, the piezoelectric coupling and the material density are constant,  \eqref{IntroductionSys1} results from choosing
$$\gamma=\frac d \rho, \qquad a=\frac C\rho+\frac{e^2}{\eps \rho} \qquad \mbox{ and }\qquad f=\frac{\Theta C B}{\rho},$$
while \eqref{HS} occurs in cases where $C$ may depend on space and time but not on the temperature $\Theta$ or the mechanical strain $u_x$. \abs
\textbf{Contribution.}\\
Building on the foundational works in \cite{Dafermos1982} and \cite{Dafermos1982a}, a plethora of studies have emerged, investigating variations of the one-dimensional problem \eqref{IntroductionSys1}, where the parameters $a$ and $\gamma$ are independent of the temperature variable. These studies have yielded noteworthy results regarding various aspects of solvability. Despite our consideration of temperature-dependent $\gamma$ and $f$, it is imperative to direct our attention to the solvability results within a broader context, particularly in regard to a more extensive range of systems.\abs
In one spatial dimension, \cite{Racke1997} demonstrated the existence of certain weak soultions to an associated initial-boundary value problem globally in time for constant $\gamma$ and broad classes of functions $a=a(u_x)$ and $F=F(u_x)$ in $f=\Theta F$. Since global smooth solutions seem even for temperature independent $\gamma$ hard to access in large data scenarios, the works \cite{Guo1999}, \cite{Shen1999} and \cite{Chen1994} rely on stronger but physically reasonable assumptions on the key ingredients. In \cite{Shibata1995} and \cite{Kim1983} the existence of global solutions for small initial data is established.\\

Higher-dimensional variations of \eqref{IntroductionSys1} have been analyzed in \cite{Roubicek2009}, \cite{Mielke2020}, \cite{Blanchard2000}, \cite{Rossi2013}, \cite{Gawinecki2016}, \cite{Gawinecki2016a}, \cite{Roubicek2013}, \cite{Pawlow2017}, and \cite{Owczarek2023} and  viscosity-free thermoelastic models, which reduce to \eqref{IntroductionSys1} with $\gamma\equiv0$ in one spatial dimension, have been addressed in \cite{Slemrod1981}, \cite{Jiang1990}, \cite{Racke1990}, \cite{Racke1991}, \cite{Racke1993}, and most recent in \cite{Rissel2020}, \cite{Cieslak2023}, \cite{Bies2023} and \cite{Bies2025}.

Recent experiments have shown a clear temperature dependence of the parameter functions \cite{Friesen2024}, thus temperature dependent $\gamma$ are not only of theoretical interest.
In this case, the analysis of solvability in the more general context poses a significant challenge. As such, classical solutions have only been obtained for simplified versions or under very specific assumptions. In particular, the existence of global classical solutions has been demonstrated for sufficiently small and smooth initial data and for constant $a$, $f\equiv0$ and a broad class of $\gamma$ (\cite{Claes}). A similar result has been achieved for related systems, considering $a\equiv b\cdot\gamma +\mu$ for arbitrary $b>0$, $\mu>0$ and assuming $F,f$ and $\gamma$ as suitably smooth. With the help of a rather technical assumption regarding special inital data, the existence of global classical solutions has also been proven (\cite{Fricke2025}).
Another strong solvability result relies on the assumption that $\gamma$ satisfies strict boundedness properties. In particular, $\gamma_0\le\gamma<\gamma_0+\delta$ is considered to hold for a certain $\delta>0$ (\cite{Winkler2025a}). Other results rely on weaker solution concepts (\cite{Winkler2025b}). 
\abs
\textbf{Main results.} \\
In this work, we focus on functions $\gamma=\gamma( \Theta)$ which are bounded, sufficiently regular, and fulfill a kind of saturation condition for $\Theta\to\infty$ and, functions $a=a(x,t)$, which we assume as sufficiently smooth and positive. For $f$ we restrict the standard assumption $f(\Theta)=\beta\Theta$ to sublinear growth by assuming $|f(\zeta)|\le C_f(1+\zeta)^{\alpha}$ with $0<\alpha<5/6$ and $|f'(\zeta)|\le C_f$.
Under these assumptions, we are able to prove the existence of global classical solutions of
\be{0}
	\lball
	u_{tt} = (\gamma(\Theta) u_{xt})_x + (a(x,t) u_x)_x + (f(\Theta))_x, 
	\qquad & x\in\Om, \ t>0, \\[1mm]
	\Theta_t = D\Theta_{xx} + \gamma(\Theta) u_{xt}^2+  f(\Theta)u_{xt},
	\qquad & x\in\Om, \ t>0, \\[1mm]
	u_x=\Theta_x=0
	\qquad & x\in\pO, \ t>0, \\[1mm]
	u(x,0)=u_0(x), \quad u_t(x,0)=u_{0t}(x), \quad \Theta(x,0)=\Theta_0(x),
	\qquad & x\in\Om.
	\ear
\ee
Considering one spatial dimension is a crucial assumption to accomplish this, since Sobolev embeddings are particularly convenient for $n = 1$, a fact that will be utilised on two occasions.
Firstly, we only require information about $\Theta$ in $W^{1,2}(\Om)$ to utilize the extensibility criterion of the local solution theory used.
Secondly, the 1-D Gagliardo-Nirenberg inequality will be applied in numerous critical instances.

\begin{theo}\label{theoglobal}
Let $\Om\subset\R$ be an open bounded interval, let \be{VorrA}
a\in C^2(\bom\times[0,\infty)) \qquad \mbox{be positive} \ee
and suppose that 
\be{VorrRegul}
 \gamma \in C^2([0,\infty)) \quad \mbox{ and } \quad f\in C^1([0,\infty))
\ee
are such that
\be{VorrGamma}
c_\gamma < \gamma(\zeta) < C_\gamma \qquad \mbox{ and } \qquad \gamma''(\zeta) \le 0 \qquad \mbox{ for all }\zeta\ge0,
\ee
and
\be{VorrF}
f(0)=0, \qquad  |f'(\zeta)|\le C_f\qquad \mbox{ and }\qquad |f(\zeta)|\le C_f(\zeta+1)^\alpha \qquad \mbox{ for all }\zeta\ge 0
\ee
for some positive constants $c_\gamma<1<C_\gamma$, $C_f>1$ and 
\be{VorrAlpha}
0<\alpha<\frac{5}{6}.
\ee
Then whenever
\be{Init}
\lball
u_0\in W^{3,2}(\Om) \mbox{ is such that }u_{0x}=0 \mbox{ on }\pO,\\[1mm]
u_{0t}\in W^{2,2}(\Om) \mbox{ is such that } u_{0tx}=0 \mbox{ on } \pO, \mbox{ and } \\[1mm]
\Theta_0\in W^{2,2}(\Om) \mbox{ satisfies }\Theta_0 \ge 0 \mbox{ in }\Om \mbox{ and } \Theta_{0x}= 0 \mbox{ on }\pO,
    \ear
\ee 
there exists a pair of functions $(u,\Theta)$ which solves \eqref{0} in the classical sense in $\Om\times(0,\infty)$, while satisfying $\Theta\ge0$ as well as
 \be{Eig1}
	\lbal
	u\in \Big( \bigcup_{\mu\in (0,1)} C^{1+\mu,\frac{1+\mu}{2}}(\bom\times [0,\infty))\Big) \cap C^{2,1}(\bom\times (0,\infty))
		\qquad \mbox{and} \\[1mm]
	\Theta\in \Big( \bigcup_{\mu\in (0,1)} C^{1+\mu,\frac{1+\mu}{2}}(\bom\times [0,\infty))\Big) 
		\cap C^{2,1}(\bom\times (0,\infty)),
	\ear
  \ee
and furthermore
\be{Eig2}
	\begin{array}{l}
	u_t\in 
	\Big( \bigcup_{\mu\in (0,1)} C^{1+\mu,\frac{1+\mu}{2}}(\bom\times [0,\infty))\Big) \cap C^{2,1}(\bom\times (0,\infty)).
	\end{array}
  \ee

\end{theo}
\section{Preliminaries}
In the following chapter, we present the results of various studies and prepare them for our particular purposes. We refrain from providing detailed proofs, instead referring to the original works.
We begin with a proposition on the existence of local solutions, which was examined in more detail in \cite[Theorem 1.1]{Fricke2025}. The proof follows from minor adjustments, since the regularity of the parameter functions $\gamma,$ $f$ and $a$ considered here is sufficient to the methods used there. Our proof of the existence of global solutions is based on the extensibility criterion contained therein.
\begin{prop}\label{Theolocal}
Let $\Om\subset\R$ be an interval and suppose $\gamma,a$ and $f$ are such that \eqref{VorrA}-\eqref{VorrAlpha} hold and $(u_0,u_{0t},\Theta_0)$ to satisfy $\eqref{Init}$. Then there exists $\tm \in(0,\infty]$ as well as a pair of functions $(u,\Theta)$
which solves \eqref{0} in the classical sense in $\Om\times(0,\tm)$ while satisfying $\Theta\ge0$ and
 \be{Eig1loc}
	\lbal
	u\in \Big( \bigcup_{\mu\in (0,1)} C^{1+\mu,\frac{1+\mu}{2}}(\bom\times [0,\tm))\Big) \cap C^{2,1}(\bom\times (0,\tm))
		\qquad \mbox{and} \\[1mm]
	\Theta\in \Big( \bigcup_{\mu\in (0,1)} C^{1+\mu,\frac{1+\mu}{2}}(\bom\times [0,\tm))\Big) 
		\cap C^{2,1}(\bom\times (0,\tm))
	\ear
  \ee
 and furthermore
  \be{Eig2loc}
	\begin{array}{l}
	u_t\in 
	\Big( \bigcup_{\mu\in (0,1)} C^{1+\mu,\frac{1+\mu}{2}}(\bom\times [0,\tm))\Big) \cap C^{2,1}(\bom\times (0,\tm)),
	\end{array}
  \ee

and which have the additional property that
  \bea{Ext}
	& & \hs{-15mm}
	\mbox{if $\tm<\infty$, \quad then \quad} 
	\limsup_{t\nearrow\tm}  \left\|\Theta(\cdot,t)\right\|_{W^{1,2}(\Om)} 
		=\infty.
  \eea
\end{prop}

In preparation for applying a Moser iteration argument later on, we record the following result which can be found proven in \cite[Lemma 2.1]{Ding2021}, for instance.      

\begin{lem}\label{MoserIt}
Let $A\ge 0$ and $B\ge1$, and suppose that $(M_k)_{k\ge 1}\subset(0,\infty)$ is such that
\bea{MoserIt1}
M_k\le \max \left\{ A^{2^k} ,B^k M_{k-1}^{2} \right\} \qquad \mbox{for all } k\ge 2.
\eea
Then 
\begin{align*}
M_k^\frac 1{2^k}\le B^2 \cdot \max\{A,M_1\} \qquad \mbox{ for all }k\ge 1. 
\end{align*}
\end{lem}

To cleanly prepare the utilization of a well-known result from the work of Porzio-Vespri in \cite[Theorem 1.3]{Porzio1993}, we note the following lemma, which represents a reduction to our homogeneous case of Theorem 1.3 therein. 
While the focus of the original work lies in establishing higher regularity properties for a given weak solution, we are interested in applying their results to already classical solutions as we want to make use of the fact that Porzio and Vespri carefully trace the dependency of the Hölder constants on other explicit bounds of said solutions. This will later allow us to establish a uniform-in-time modulus of continuity for our solution regarding their space variable and thus enable a crucial Ehrling-type inequality from the literature.

\begin{lem}\label{PorzioVespriLit}
Let $T>0$, $q>3$ and $w_0\in W^{2,2}(\Om)$, and suppose 
$w\in C^{2,1}(\bom\times(0,T))\cap L^\infty(\Om\times(0,T))$ to be a classical solution of
\be{PVSystem}
	\lball
	w_{t} =\big[ g_1(x,t)w_{x}+g_2(x,t)+g_3(x,t)\big]_x & x\in\Om,\quad t>0,\\[1mm]
	w_x=0 & x\in \pO, \quad t>0,\\[1mm]
	w(x,0)=w_0(x), 
	\qquad & x\in\Om,
	\ear
\ee
 where $c_g\le g_1\le C_g$, $g_2\in L^\infty((0,T);L^2(\Om))$ and $g_3\in L^{q}(\Om\times(0,T))$. 
Then there exist constants $\mu>1$ and $\beta\in(0,1)$ such that 
$$|w(x_1,t_1)-w(x_2,t_2)|\le \mu \left(|x_1-x_2|^\beta+|t_1-t_2|^{\frac \beta2}\right)$$
for every  pair of points $(x_1,t_1),(x_2,t_2)\in\bom\times(0,T)$. The constants $\mu>1$ and $\beta\in(0,1)$ depend upon $\|w\|_{L^\infty(\Om\times[0,T])}$ on $c_g$,  $C_g$, $\|g_2\|_{L^{\infty}((0,T);L^2(\Om))}$,  $\|g_3\|_{L^{q}(\Om\times(0,T))}$ and the norm $\|w_0\|_{C^0(\Om)}$.
\end{lem}


In line with the former, we introduce another lemma which derives a useful Ehrling-type inequality from the given Hölder continuity. For the proof, we refer to \cite[Lemma A.1]{Tao2019}.
\begin{lem}\label{A1}
Let $\Om\subset\R$ be an interval, and let $\omega:(0,\infty)\to(0,\infty)$ be nondecreasing. Then for all $p\ge1$ and each $\eta>0$ there exists $C(p,\eta)>0$ such that
\begin{align}
    \io |\vp_x|^{2p+2} \le \eta \io |\vp_x|^{2p-2}\vp_{xx}^2+C(p,\eta)\|\vp\|_{L^\infty(\Om)}^{2p+2}
\end{align}
holds for all 
\bea{A1GL}
\vp \in S_\omega:= \Bigg\{ \tilde\vp \in C^2(\bom)\Big| & & \frac{\partial \tilde\vp}{\partial \nu}=0\mbox{ on }\pO, \mbox{ and for all $\eta'>0$,}\nn\\
& & \mbox{ we have $|\tilde\vp(x)-\tilde\vp(y)|<\eta'$}\nn\\ & & \mbox{ whenever $x,y\in \bom$ are such that $|x-y|<\omega(\eta')$}\Bigg\}.
\eea
\end{lem}

\section{Global existence time. Proof of Theorem \ref{theoglobal}}
In this work, whenever $\gamma,a,f$ and $(u_0,u_{0t},\Theta_0)$ are fixed such that \eqref{VorrA}-\eqref{Init} are met, we let  $\tm$, $u$ and $\Theta$ be as found in Proposition \ref{Theolocal}.\\
With our toolbox full, we are now set out to prove Theorem 1.1, by deriving a bound for $\Theta$ in $W^{1,2}(\Om)$. To do so, we will start with simple testing procedures to obtain information about the mechanical energy and the average temperature; in particular, we derive bounds for $u_x$ and $u_t$ in $L^2(\Om)$ and for $\Theta$ in $L^1(\Om)$. 
\begin{lem}\label{Mass}
Let $\gamma,a$ and $f$ be such that $\eqref{VorrA}$-\eqref{VorrAlpha} are satisfied and assume $(u_0,u_{0t},\Theta_0)$ to fulfill $\eqref{Init}$. Then for every $T>0$, there exists $C_1=C_1(a,T)>1$ such that
\bea{C1}
\io u_t^2 \le C_1, \qquad \io u_x^2 \le C_1, \qquad\mbox{and}\qquad \io \Theta \le C_1
\eea
for all $t\in(0,T)\cap(0,\tm)$.
\end{lem}

\proof
Since $a\in C^2(\bom\times[0,\infty))$ is positive, there exists $c_1= c_{1}(a,T)>0$ such that $|a_t(x,t)|\le c_1$ and $\frac{1}{c_1}\le a(x,t)\le c_1$ on $\Om\times[0,T]$. By simple testing procedures, with respect to $u_x=u_{tx}=\Theta_x=0$ on $\pO$, we obtain
\begin{align}\label{MassRe1}
    \frac 12 \dt \io u_t^2 +\frac 12 \dt \io a(x,t)u_x^2 + \dt \io \Theta  &= -\io \gamma(\Theta)u_{xt}^2-\io a(x,t)u_x u_{xt}-\io f(\Theta) u_{xt}+\frac 12 \io a_t(x,t) u_x^2\nn\\
    &\quad +\io a(x,t)u_x u_{xt}
    + \io D\Theta_{xx} +\io\gamma(\Theta)u_{xt}^2+\io f(\Theta) u_{xt}\nn\\
    &=  \ \frac12 \io a_t(x,t) u_x^2\nn \\
    &\le c_{1}^2 \io a(x,t) u_x^2\qquad \mbox{for all } t\in(0,T)\cap(0,\tm).\end{align}
We define
 $$
   y(t):= \frac{1}{2}\io u_t^2 +\frac12 \io a(x,t)u_x^2 +\io \Theta\qquad \mbox{ for all }t\in(0,T)\cap(0,\tm),$$ and hence from \eqref{MassRe1}, we derive the differential inequality
   $$    y'(t)\le c_{1}^2 \cdot y(t)\qquad \mbox{ for all }t\in(0,T)\cap(0,\tm),$$
    and since
    $$y(0)= \frac{1}{2}\io u_{0t}^2 +\frac12  \io a(\cdot,0)u_{0x}^2 +\io \Theta_0 <\infty,$$
    due to $\eqref{Init}$, we infer by a comparison argument
    $$y(t)\le  y(0)  \cdot e^{c_{1}^2T} \qquad \mbox{for all }t\in(0,T)\cap(0,\tm),$$ which already establishes our claim \eqref{C1} due to the nonnegativity of all the components of $y$. 
    \qed
Now, in view of our overall assumption that $f$ satisfies \eqref{VorrF} with $\alpha<5/6<1$, we can control the ill-signed contributions to the evolution of functionals $\io (\Theta+1)^p$ in their concave range, where $p\in (0,1)$.

\begin{lem}\label{Theta1}
Let $\gamma,a$ and $f$ be such that $\eqref{VorrA}$-\eqref{VorrAlpha} are satisfied and assume $(u_0,u_{0t},\Theta_0)$ to fulfill $\eqref{Init}$. For all $p\in(0,1)$ and all $T_0>0$, there exists $C_2=C_2(\gamma,a,f,\Om,p,T_0)>0$, such that
\bea{c2}
\int_0^T \io (\Theta+1)^{p-2} \Theta_x^2 \le C_2\qquad & & \mbox{ for all } T\in(0,T_0)\cap(0,\tm) \label{c21}
\eea
\end{lem}
\proof 
We begin by fixing $p\in(0,1)$ and $T\in(0,T_0)\cap(0,\tm)$ and testing the second equation of $\eqref{0}$ to obtain 
\begin{align*}-\frac 1 p \dt \left\{\io (\Theta+1)^p\right\} +(1-p)D\io (\Theta+1)^{p-2}\Theta_x^2 &= -\io \gamma(\Theta) (\Theta+1)^{p-1} u_{xt}^2-\io (\Theta+1)^{p-1} f(\Theta) u_{xt} \\
&\le-\frac {1}2 \io \gamma(\Theta) (\Theta+1)^{p-1} u_{xt}^2 +\frac{C_f^2}{2c_\gamma}\io (\Theta+1)^{p-1+2\alpha}\end{align*}
for all $t\in(0,T)$, where we used our assumed boundary condition $\Theta_x=0$.

In preparation of a further estimation step, we first apply Young's inequality to conclude that since $\alpha\in(0,1)$, there exists $c_1= c_1(\gamma,f,\Om,p)>0$ such that
$$
\frac{C_f^2}{2c_\gamma}\io (\Theta+1)^{p-1+2\alpha}\le  c_1+\io (\Theta+1)^{p+1} \qquad \mbox{ for all }t\in(0,T),$$
and by an application of a Gagliardo-Nirenberg inequality, we infer that there exists $c_2= c_2(\Om,p)>0$ such that
\bea{bvorbereitung} \Big\|\vp \Big\|_{L^{\frac {2(p+1)} p}(\Om)}^{\frac{2(p+1)}{p}}
\le c_2\Big\|\vp_x\Big\|_{L^2(\Om)}^{\frac{2(p+1)\lam}{p}}\Big\|\vp \Big\|_{L^{\frac 2 p}(\Om)}^{(1-\lam)\frac{2(p+1)}p}+c_2\Big\| \vp \Big\|_{L^{\frac 2p}(\Om)}^{\frac{2(p+1)}{p}} \nn
\eea 
 for every $\vp\in C^1(\bom)$ with $\lam:=\frac{p^2}{(p+1)^2}\in(0,1)$
and when applied to $\vp=(\Theta+1)^{\frac{p}{2}}$ we infer
\bea{b} \io (\Theta+1)^{p+1} &=&\left\|(\Theta+1)^{\frac p2}\right\|_{L^{\frac {2(p+1)} p}(\Om)}^{\frac{2(p+1)}{p}}\nn\\
&\le& c_2\left\|\left((\Theta+1)^{\frac p2}\right)_x\right\|_{L^{2}(\Om)}^{\frac{2p}{p+1}}\left\|(\Theta+1)^{\frac p2}\right\|_{L^{\frac 2 p}(\Om)}^{(1-\lam)\frac{2(p+1)}p}+c_2\left\|(\Theta+1)^{\frac p2}\right\|_{L^{\frac 2p}(\Om)}^{\frac{2(p+1)}{p}} \nn
\eea 
for all $t\in(0,T)$. Moreover, since
$\frac{(p+1)\lam}{p}=\frac p{p+1} <1$, an application of Young's inequality and Lemma \ref{Mass} leads to
\begin{align*} c_2\left\|\left((\Theta+1)^{\frac p2}\right)_x\right\|_{L^{2}(\Om)}^{\frac{2p}{p+1}}\left\|(\Theta+1)^{\frac p2}\right\|_{L^{\frac 2 p}(\Om)}^{(1-\lam)\frac{2(p+1)}p}&\le \frac{(1-p)D}2 \io (\Theta+1)^{p-2} \Theta_{x}^2+\frac{2^pc_2^{p+1}}{(1-p)^pD^p} \left(\io \Theta+1\right)^{2p+1}\\
&\le \frac{(1-p)D}2 \io (\Theta+1)^{p-2} \Theta_x^2+\frac{2^pc_2^{p+1}}{(1-p)^pD^p}  \left(C_1+|\Om|\right)^{2p+1}\\
&\le \frac{(1-p)D}2 \io (\Theta+1)^{p-2} \Theta_x^2 +c_3
\end{align*}
 for all $t\in(0,T)$ with $c_3\equiv c_3(\gamma,a,f,\Om,p):=\left(\frac2{(1-p)D}\right)^p  c_2^{p+1}\left(C_1+|\Om|\right)^{2p+1}>0$.
Combining these two leads to the conclusion that
\begin{align*}-\frac  1p \dt \io (\Theta+1)^p +\frac{(1-p)D}2\io (\Theta+1)^{p-2}\Theta_x^2 &\le  c_1 +c_3 + c_2\left(C_1+|\Om|\right)^{p+1}\\
&=c_4
\end{align*}
for all $t\in(0,T)$, where $c_4\equiv c_4(\gamma,a,f,\Om,p):=c_1 +c_3 +c_2\left(C_1+|\Om|\right)^{p+1}>0$. 
Integrating in time over $(0,T)$ yields
\begin{align*}\frac{(1-p)D}2\int_0^T\io (\Theta+1)^{p-2}\Theta_x^2
&\le c_4T+\frac 1 p \io (\Theta(\cdot,T)+1)^p-\frac 1p\io (\Theta(\cdot,0)+1)^p \\
&\le c_4T+ \frac{C_1}{p}+\frac2p|\Om|,
\end{align*}
 which establishes \eqref{c21}. 
\qed

Now we exploit the fact that we consider $\Om\subset \R$ to be of one space dimension once again by using our previously derived information in the context of another Gagliardo-Nirenberg inequality to obtain information about $\int_0^T\| \Theta\|_{L^q(\Om)}^r$ for suitable $q,r>1$. It is important to note that the argument used in this context is not applicable to higher spatial dimensions.
The conditions on $r$ with respect to $q$ that are determined in this process will later clarify the restriction $\alpha<5/6$ when applied to $f$. 
\begin{lem}\label{Thetah3}
Let $\gamma,a$ and $f$ be such that $\eqref{VorrA}$-\eqref{VorrAlpha} are satisfied and assume $(u_0,u_{0t},\Theta_0)$ to fulfill $\eqref{Init}$.
For any $T_0>0$,  $q>1$ and $r>0$ which satisfy
$$r<\frac{2q}{q-1},$$ there exists  $C_3=C_3(q,r,\gamma,a,f,\Om,T_0)>0$ such that 
\bea{Thetah3rech}
\int_0^T\big\| \Theta+1\big\|_{L^q(\Om)}^{r} &\le& C_3 \qquad \mbox{ for all } T\in(0,T_0)\cap(0,\tm) 
\eea

\end{lem}

\proof

We fix $T\in(0,T_0)\cap(0,\tm)$, $q>1$ and $p\in(0,1)$. For all $\vp \in C^1(\bom)$ and every $r\in(0,1]$ using Young's inequality, we conclude that
$$\|\vp\|_{L^q(\Om)}^{r} \le \|\vp \|_{L^q(\Om)}^{r'}+1$$
for any $r'>1$. It is therefore evident that the subsequent analyzes will exclusively address the case of $r>1$. We fix $r>1$ and infer  by an application of a Gagliardo-Nirenberg inequality that there exists some $c_1= c_1(r,q,p,\Om)>1$ such that
$$\big\|\vp\big\|_{L^{\frac{2q}{p}}(\Om)} \le c_1 \big\| \vp_x\big\|_{L^2(\Om)}^\lam\big\|\vp\big\|_{L^{\frac{2}{p}}(\Om)}^{(1-\lam)}+c_1\big\|\vp\big\|_{L^{\frac2p}(\Om)} \qquad \mbox{ for all }\vp \in C^1(\bom)$$
with $\lam:=\frac{p(q-1)}{(p+1)q}\in(0,1)$. When applied to $\vp=(\Theta+1)^{\frac{p}{2}}$, we obtain
\bea{Thetah3Gl1}
 \Big\|\Theta+1\Big\|_{L^q(\Om)}^r &=&  \left\|(\Theta+1)^{\frac p2}\right\|_{L^{\frac{2q}p}(\Om)}^{\frac{2r}p}\nn\\
&\le& c_1\left\| \left((\Theta+1)^{\frac p2}\right)_x\right\|_{L^2(\Om)}^{\frac{2r}p \lambda}\left\|(\Theta+1)^{\frac p2}\right\|_{L^{\frac 2p}(\Om)}^{\frac{2r}p(1-\lambda)} + c_1\left\|(\Theta+1)^{\frac p2}\right\|_{L^{\frac 2p}(\Om)}^{\frac{2r}p}\nn
\eea
for all $t\in(0,T)$. Since $r<\frac{2q}{q-1}$, another application of Young's inequality, when combined with Lemma \ref{Mass} and Lemma \ref{Theta1}, yields

\bea{Thetah3Gl2}
c_1\int_0^T \left\| \left((\Theta+1)^{\frac p2}\right)_x\right\|_{L^2(\Om)}^{\frac {2r (q-1)}{(p+1)q}} \left\| (\Theta+1)^{\frac p2}\right\|_{L^\frac{2}{p}(\Om)}^{\frac{2r}{p}(1-\lambda)}&\le& \int_0^T \io (\Theta+1)^{p-2} \Theta_x^2 +c_1^{\frac{(p+1)q}{(p+1)q-r(q-1)}}\int_0^T \left(\io \Theta+1 \right)^{\tilde q} \nn\\
&\le&c_2,\nn 
\eea
where $\tilde q:=\frac{rq+rp}{pq+q-rq+r}$ and $c_2\equiv c_2(\gamma,a,f,\Om,T_0):=C_2  + c_1^{\frac{(p+1)q}{(p+1)q-r(q-1)}}(C_1+|\Om|)^{\tilde q} T_0$.
Since Lemma \ref{Mass} also implies that
 $$c_1\int_0^T\left\|(\Theta+1)^{\frac p2}\right\|_{L^{\frac 2p}(\Om)}^{\frac{2r}p}\le c_1( C_1+|\Om|)^rT_0$$
holds, we have already accomplished \eqref{Thetah3rech}.
\qed

In order to apply Lemma \ref{PorzioVespriLit}, we require $L^\infty(\Om)$ bounds, which are finite for $\tm<\infty$. To derive these bounds, we use a limit process based on reducing higher-degree $L^{p_k}$-norms to $L^{p_{k-1}}$-norms, employing a Moser-type iteration and in particular Lemma \ref{MoserIt}.

\begin{lem}\label{MoserAnw}Let $\gamma,a$ and $f$ be such that $\eqref{VorrA}$-\eqref{VorrAlpha} are satisfied for some $\alpha<5/6$ and assume $(u_0,u_{0t},\Theta_0)$ to fulfill $\eqref{Init}$. Then for all $T_0>0$, there exists $C_4=C_4(\gamma,a,f,\Om,T_0)>0$ such that
    \bea{Moser}
    \|u_t(\cdot,t)\|_ {L^\infty(\Om)}\le C_4 \qquad \mbox{ for all }t\in(0,T_0)\cap(0,\tm).   
    \eea
\end{lem}

\proof We fix $T\in(0,T_0)\cap(0,\tm)$. For integers $k\ge1$, we let $p_k:=2^k$, $I:=(0,T)\cap(0,\tm)$ and
\bea{MoserKonst}
M_k(T):= \ \sup_{t \in I} \io u_t^{p_k}(\cdot,t)<\infty.
\eea
 Furthermore, since with respect to $\al<5/6$, we are able to choose $q=q(\al)\in(2,\frac1{(3\al-2)_+})$ and define $p=p(\al):=\frac{2q}{q-2}$. Due to $a$ satisfying $\eqref{VorrRegul}$, we may utilize Lemma \ref{Mass} to obtain $c_1=c_1( a,T_0)>1$ such that
\bea{Moserc1Schranke}
\io a^2 u_x^2 \le c_1, \qquad  \mbox{ for all } t\in I.
\eea
We now apply another Gagliardo-Nirenberg inequality to infer that there exists $c_2=c_2(\Om)>0$ such that
\bea{MoserGag}
\big\| \vp \big\|_{L^p(\Om)}^2 \le c_2 \big\| \vp_x\big\|_{L^2(\Om)}^{2\lam}\big\|\vp\big\|_{L^1(\Om)}^{2(1-\lam)}+c_2\big\|\vp\big\|_{L^1(\Om)}^2
\eea 
and 
\bea{MoserGag2}
\big\| \vp \big\|_{L^\infty(\Om)}^2 \le c_2 \big\| \vp_x\big\|_{L^2(\Om)}^{\frac{4}{3}}\big\|\vp\big\|_{L^1(\Om)}^{\frac{2}{3}}+c_2\big\|\vp\big\|_{L^1(\Om)}^2
\eea 
for all $\vp\in C^1(\bom)$, with $\lam=\lam(q,\al):=\frac{q+2}{3q}$. By testing the first equation and adding $\int_\Om u_t^{p_k}$ on both sides, we obtain 
\bea{Moser1}
 \dt \io u_t^{p_k} + \io u_t^{p_k} &=& p_k\io u_t^{p_k-1} \big[\gamma(\Theta)u_{tx}+au_x+f(\Theta) \big]_x+\io u_t^{p_k}\nn\\
 &=& - p_k(p_k-1) \io  u_t^{p_{k}-2} u_{tx} \big[\gamma(\Theta)u_{tx}+a u_x+f(\Theta)\big]+\io u_t^{p_k}\nn\\
 &\le& -\frac{p_k(p_k-1)c_\gamma}2\io  u_t^{p_k-2} u_{tx}^2+\frac{p_k(p_k-1)}{c_\gamma}\io u_t^{p_k-2}f(\Theta)^2 \nn\\
 & &+\frac{p_k(p_k-1)}{c_\gamma}\io u_t^{p_k-2} a^2u_x^2 +\frac{p_k(p_k-1)c_1}{c_\gamma}\io u_t^{p_k}
 \eea
for all $t\in I$, where we have used the facts that $c_1>1$, $c_\gamma<1$ and $u_{tx}=0$. To estimate the right hand side, we start by analyzing the two last terms, therefore we note 
$\|u_t^{p_k-2}\|_{L^\infty(\Om)}\le\|u_t^{\frac{p_k}2}\|^2_{L^\infty(\Om)}+1 $ and we conclude by applying  \eqref{MoserGag2} to $\vp=u_t^{\frac{p_k}{2}}$ and Young's inequality for $c_3\equiv c_3(\Om):=c_2(1+|\Om|)$
\bea{Moser2}
    (1+|\Om|)\left\|u_t^{\frac{p_k}{2}}\right\|_{L^\infty(\Om)}^2    &\le&  c_3\left\|\left(u_t^{\frac{p_k}2}\right)_x\right\|_{L^2(\Om)}^{\frac43}\left\|u_t^{\frac{p_k}2}\right\|_{L^1(\Om)}^{\frac{2}{3}}+c_3\left\|u_t^{\frac{p_k}2}\right\|_{L^1(\Om)}^2\nn\\
&\le&    \frac{c_\gamma^2}{4 c_1} \io u_t^{p_k-2} u_{tx}^2 +\left(16\frac{c_1^2 c_3^3}{c_\gamma^4}+c_3\right) \left(\io u_t^{\frac{p_k}{2}} \right)^{2}\nn\\
 &=& \frac{c_\gamma^2}{4c_1} \io u_t^{p_k-2} u_{tx}^2 +c_4 \left(\io u_t^{\frac{p_k}{2}} \right)^2,
\eea
for all $t\in I$, where $c_4\equiv c_4(\gamma,a,f,\alpha,M,T_0):= 16\frac{c_1^2c_3^3}{c_\gamma^4}+c_3$. This, when combined with \eqref{Moserc1Schranke} and the facts that $c_1>1$ and $c_\gamma<1$, leads to
\bea{Moser2terTeil} \frac{p_k(p_k-1)}{c_\gamma}\left(\io u_t^{p_k-2} a^2u_x^2 + c_1\io u_t^{p_k}\right)&\le&\frac{p_k(p_k-1)}{c_\gamma}\left[\left(\left\|u_t^{\frac{p_k}2}\right\|_{L^\infty(\Om)}^2+1\right)\io a^2 u_x^2+|\Om|c_1\left\|u_t^{\frac{p_k}2}\right\|_{L^\infty(\Om)}^2  \right]\nn\\
&\le& \frac{p_k(p_k-1)c_1}{c_\gamma} \left[1+(1+|\Om|)\left\|u_t^{\frac{p_k}{2}}\right\|_{L^\infty(\Om)}^2\right]\nn\\
&\le& \frac{p_k(p_k-1)c_\gamma}{4} \io u_t^{p_k-2} u_{tx}^2 +\frac{p_k^2c_1c_4}{c_\gamma} \left(\io u_t^{\frac{p_k}{2}} \right)^2+\frac{p_k^2c_1}{c_\gamma} 
\eea
for all $t\in I$. Furthermore, for the last term in $\eqref{Moser1}$, we note the observation that
$$\io u_t^{p_k-2} f(\Theta)^2\le \io u_t^{p_k} f(\Theta)^2+ \io f(\Theta)^2  \qquad \mbox{for all }t\in I,$$
and infer due to \eqref{MoserGag},  Hölder's inequality and Young`s inequality 
\bea{Moser21Alt}
\io u_t^{p_k} f(\Theta)^2 &\le& \left\|  u_t^{\frac{p_k}2} \right\|_{L^{p}(\Om)}^2 \Big\| f(\Theta)\Big\|_{L^{q}(\Om)}^2\nn\\
&\le& \left(c_2\left\|\left(u_t^{\frac{p_k}{2}}\right)_x\right\|_{L^2(\Om)}^{2\lam} \left\|u_t^{\frac{p_k}{2}}\right\|_{L^1(\Om)}^{2(1-\lam)}+c_2 \left\|u_t^{\frac{p_k}{2}}\right\|_{L^1(\Om)}^2\right)\Big\| f(\Theta)\Big\|_{L^{q}(\Om)}^2 \nn\\
&\le& \frac{c_\gamma^2}{4}\left\|\left(u_t^{\frac{p_k}{2}}\right)_x\right\|_{L^2(\Om)}^{2} + \left(\frac{2}{c_\gamma}\right)^{\frac{2\lam}{1-\lam}}c_2^{\frac{1}{1-\lam}}\left\|u_t^{\frac{p_k}{2}}\right\|_{L^1(\Om)}^{2}  \Big\| f(\Theta)\Big\|_{L^{q}(\Om)}^{\frac{2}{1-\lam}}+c_2 \left\|u_t^{\frac{p_k}{2}}\right\|_{L^1(\Om)}^2\Big\| f(\Theta)\Big\|_{L^{q}(\Om)}^2\nn
\eea
for all $t\in I.$ We can now draw on $\eqref{Moserc1Schranke}$ and the fact that $A^2 \le A^{\frac{2}{1-\lam}}+1$ holds for all $A\ge0$ and $\lam=\frac{q+2}{3q}\in(0,1)$ as well as $\|u_t^{\frac{p_k}{2}}\|_{L^1(\Om)}^2\le M_{k-1}^2$ for all $t\in I$ to obtain for  $c_5\equiv c_5(q,\gamma,a,f,\alpha,M,T_0):=c_2^{\frac{1}{1-\lam}}\left(\frac{2}{c_\gamma}\right)^{\frac{2\lam}{1-\lam}}+c_2$, that
\bea{Moser22Alt}
\frac{p_k(p_k-1)}{c_\gamma}\io u_t^{p_k-2}f(\Theta)^2 \le \frac{p_k(p_k-1)c_\gamma}4 \io u_t^{p_k-2} u_{tx}^2 +\frac{ p_k^2 c_5}{c_\gamma}M_{k-1}^2 \Big\| f(\Theta)\Big\|_{L^{q}(\Om)}^{\frac{2}{1-\lam}} + \frac{p_k^2c_2}{c_\gamma} M_{k-1}^2
\nn\eea
for all $t\in I$. If we now insert this together with \eqref{Moser2terTeil} into \eqref{Moser1}, we may conclude that
\bea{MoserVorletzterTeil}
\dt \io u_t^{p_k} + \io u_t^{p_k} &\le&  \frac{ p_k^2 c_5}{c_\gamma}M_{k-1}^2 \Big\| f(\Theta)\Big\|_{L^{q}(\Om)}^{\frac{2}{1-\lam}} + \frac{p_k^2(c_2+c_1c_4)}{c_\gamma} M_{k-1}^2 +\frac{p_k^2c_1}{c_\gamma} +\frac{p_k^2}{c_\gamma}\io f(\Theta)^2
\eea
for all $t\in I.$ Noting that $q\in(2,\frac{1}{(3\alpha-2)_+})$ and $\lam=\frac{q+2}{3q}$ ensures $\frac{2\alpha}{1-\lam}=\frac{3q\alpha}{q-1}<\frac{2q\alpha}{q\alpha-1}$, we may apply Lemma \ref{Thetah3} to infer for some $c_6=c_6(\gamma,a,f,M,\Om,T_0)>0$
\be{alphaAnwendung}
\int_0^T\Big\| f(\Theta)\Big\|_{L^{q}(\Om)}^{\frac{2}{1-\lam}} \le C_f^{\frac{2}{1-\lam}} \int_0^T \Big\| \Theta+1\Big\|_{L^{q\alpha}(\Om)}^{\frac{2\alpha}{1-\lam}} \le c_6
\ee
and
$$\int_0^T\Big\| f(\Theta)\Big\|_{L^{2}(\Om)}^2 \le C_f^2\int_0^T \Big\| \Theta+1\Big\|_{L^{2\al}(\Om)}^{2\al} \le c_6.  $$
Utilizing this in combination with the fact that $\io u^{p_k}(\cdot,0)\le |\Om| \cdot \|u_{0t}\|_{L^\infty(\Om)}^{p_k}$ further by integrating \eqref{MoserVorletzterTeil} in time over $(0,t)$ to infer
\begin{align*}\io u_t^{p_k}(\cdot,t) &+ \int_0^t \io u_t^{p_k} \nn\\
&\le \frac{p_k^2 c_5}{c_\gamma} M_{k-1}^2 \int_0^T \Big\| f(\Theta)\Big\|_{L^{q}(\Om)}^{\frac{2}{1-\lam}}+ \frac{p_k^2(c_2+c_1c_4)}{c_\gamma} M_{k-1}^2 T_0 +\frac{p_k^2c_1}{c_\gamma}T_0 +\frac{p_k^2}{c_\gamma}\int_0^T\io f(\Theta)^2 + \io u_t^{p_k}(\cdot,0) \\
&\le \frac{p_k^2}{c_\gamma}(c_5c_6+c_2T_0+c_1c_4T_0) M_{k-1}^2 +\frac{p_k^2}{c_\gamma}(c_1T_0+c_6)+|\Om| \cdot\|u_{0t}\|_{L^\infty(\Om)}^{p_k}\\
&\le p_k^2 c_7 M_{k-1}^2 + p_k^2 c_8 + |\Om|\cdot \|u_{0t}\|_{L^\infty(\Om)}^{p_k}
\end{align*}
for all $t\in I$ with $c_7\equiv c_7(\gamma,a,f,M,\Om,T_0):= \frac{c_5c_6+c_2T_0+c_1c_4T_0}{c_\gamma} $ as well as  $c_8=c_8(\gamma,a,f,M,\Om,T_0):=\frac{c_1T_0+c_6}{c_\gamma}$ and thus 
$$M_k \le 2^{2k} c_7 M_{k-1}^2 + 2^{2k} c_8 + |\Om|\cdot \|u_{0t}\|_{L^\infty(\Om)}^{2^k}.$$
Since, by \eqref{Init}, $\|u_{0t}\|_{L^\infty(\Om)}$ is finite and by $A:= \max\big\{(1+|\Om|) \cdot\| u_{0t}\|_{L^\infty(\Om)}+2(1+c_8)\big\}$ as well as $B:=4\cdot\max\{c_7,1\}$, we are able to apply Lemma \ref{MoserIt} to conclude from
$$M_k \le B^k  M_{k-1}^2 +A^{2^k} \le \max\{ A^{2^k},B^k M_{k-1}^2\},\qquad \mbox{ for all }k\ge2,$$
in view of $p_k:=2^k$
$$M_k^{\frac{1}{2^k}}\le B^2\max\big\{ A, M_1 \big\}, \qquad\mbox{ for all }k\ge 1,$$
where the independence of $M_1=\sup_{s\in I}\|u_t(\cdot,s)\|_{L^2(\Om)}^2$ from $T$ and it's finiteness are ensured by Lemma \ref{Mass}. Taking the limits $k\to\infty$ and $T\nearrow \tm$ establishes \eqref{Moser}.
\qed

As demonstrated in the preceding proof, it is evident that restricting $\alpha<5/6$ is indispensable for our line of argument. Without it, we could not in general select $q\in(2,\frac{1}{(3\al-2)_+})$ as required to invoke Lemma \ref{Thetah3} in \eqref{alphaAnwendung}.\abs
We have now demonstrated all the prerequisites to derive the Hölder continuity of $u_t$ from Lemma \ref{PorzioVespriLit}. 
\begin{lem}\label{porzioVesp}
Let $\gamma,a$ and $f$ be such that \eqref{VorrA}-\eqref{VorrAlpha}  are satisfied and assume $(u_0,u_{0t},\Theta_0)$ to fulfill $\eqref{Init}$. Then for any $T_0>0$ there exist $\beta\in(0,1)$ and $C_5=C_5\left(\gamma,a,f,\Om,\beta,T_0\right)>0$ such that
\be{utHölder}| u_t(x_1,t_1)-u_t(x_2,t_2)| \le C_5\left( |x_1-x_2|^\beta+|t_1-t_2|^{\frac \beta 2}\right)  \ee
for every pair of points $(x_1,t_1),(x_2,t_2)\in  \Om\times\big((0,T_0)\cap(0,\tm)\big)$
\end{lem}
\proof
We fix $T\in(0,T_0)\cap(0,\tm)$, $\alpha_0\in(\alpha,1)$ and $q=\frac{3}{\alpha_0}$ and may infer by Lemma \ref{Mass} and $a\in C^2(\bom\times[0,\infty))$ as well as \eqref{VorrF}, Lemma \ref{Thetah3} and the fact that $\frac{3\alpha}{\al_0}<3$, we can infer that there exist constants $c_1=c_1(a,T_0)>0$ and $c_2=c_2(\gamma,f,\Om,T_0)>0$ such that 
$$\|a(x,t)u_x(x,t)\|_{L^2(\Om)}\le c_1 \quad \mbox{ for all }t\in(0,T) \mbox{ and }\quad \int_0^T\|f(\Theta)\|_{L^q(\Om)}^q \le  C_f^{\frac{3\al}{\al_0}}\int_0^T \io(\Theta+1)^{\frac{3\alpha}{\al_0}} \le c_2.$$  Furthermore, recalling $u_{0t}\in W^{2,2}(\Om)$  and Lemma \ref{MoserAnw}, which yields that there exists $c_3=c_3(\gamma,a,f,\Om,T_0)>0$ such that
 $$\|u_t(\cdot,t)\|_{L^\infty(\Om)}\le c_3$$
 for all $t\in(0,T)$, we may apply Lemma \ref{PorzioVespriLit} on $w=u_t$, $w_0=u_{0t}$, $g_1(x,t)=\gamma(\Theta(x,t))$, $g_2(x,t)=a(x,t)u_x(x,t)$ and $g_3(x,t)=f(\Theta(x,t))$ to infer that there exist $\beta\in(0,1)$ and $\mu(\gamma,a,f,\Om,\beta,T_0)>0$  such that \eqref{utHölder} holds. \qed

Having gathered all the information we need, we are now in a position to take the final step to prove Theorem \ref{theoglobal}.

\begin{lem}\label{großy}
Let $\gamma,a$ and $f$ be such that \eqref{VorrA}-\eqref{VorrAlpha}  are satisfied and assume $(u_0,u_{0t},\Theta_0)$ to fulfill $\eqref{Init}$. For every $T_0>0$ there exists $C_5=C_5(\gamma,a,f,\Om,T_0)>0$ such that
\bea{großy1}
\io \Theta_{x}^2 &\le& C_6\qquad\mbox{ for all }t\in(0,T_0)\cap(0,\tm).\nn
\eea
\end{lem}
\proof
We start by fixing $T_0>0$. Since $\gamma$ satisfies $\gamma''(\zeta)\le0$ and $c_\gamma<\gamma(\zeta)$ on $[0,\infty)$ and since $\Theta$ is non-negative, there exists $c_1=c_1(\gamma)>0$ such that
\bea{c1Schranke}|\gamma'(\zeta)|\le c_1\qquad\mbox{ on } [0,\infty).\eea
Furthermore, due to $a\in C^2(\bom\times[0,\infty))$ and $a$ being positive, one can find $c_2=c_2(a,T_0)>0$ such that
\bea{c2Schranken} \frac{1}{c_2} \le a(x,t), \qquad  |a_t(x,t)|\le c_2 \qquad \mbox{ and }\qquad |a_x(x,t)|\le c_2 \qquad \mbox{ on }\bom\times [0,T_0].\eea
Firstly, we note due to $0=u_x=u_{tx}$ on $\pO$  
\begin{align}\label{HilfeYoung1}
-\io u_{txx} \gamma'(\Theta)\Theta_x u_{tx}+\Big[u_{tx}^2 \gamma'(\Theta)\Theta_x\Big]_{\pO}
&= \io u_{tx} \big(\gamma'(\Theta) \Theta_x u_{tx}\big)_x\nn\\
    &=   \io \gamma''(\Theta)\Theta_x^2 u_{tx}^2+\io \gamma'(\Theta) \Theta_{xx} u_{tx}^2 + \io \gamma'(\Theta) \Theta_x u_{txx}u_{tx}
    \end{align}
for all $t\in (0,T_0)\cap(0,\tm)$, from which we conclude
\begin{align}\label{HilfeYoung2}
-\io \gamma'(\Theta) \Theta_x u_{txx} u_{tx}
    &=   \frac{1}{2}\io \gamma''(\Theta)\Theta_x^2 u_{tx}^2+\frac{1}{2}\io \gamma'(\Theta) \Theta_{xx} u_{tx}^2 
    \end{align}
for all $t\in (0,T_0)\cap (0,\tm)$. Now, we estimate by testing the first equation of \eqref{0} and an application of Young's inequality and the previously noted \eqref{c1Schranke}-\eqref{HilfeYoung2}
\bea{c}
\frac12 \dt \io u_{tx}^2 &=& \io u_{tx}u_{ttx}\nn\\
                             &=&  -\io u_{txx} 
                                \cdot\Big[\gamma(\Theta) u_{txx}+a(x,t)u_{xx}+a_x(x,t)u_x +f'(\Theta)\Theta_x\Big]+ \io u_{tx}\Big(\gamma'(\Theta)\Theta_x u_{tx}\Big)_x\nn\\
                            &\le&  -\io \gamma(\Theta) u_{txx}^2  
                                 -\frac12\dt \io a(x,t) u_{xx}^2+ \frac12 \io a_t(x,t) u_{xx}^2+\frac{c_2^2}{c_\gamma}\io  u_x^2 +\frac{C_{f'}^2}{c_\gamma} \io \Theta_x^2 \nn\\ & & +\frac{c_\gamma}{2}\io u_{txx}^2 +\frac12\io \gamma''(\Theta)\Theta_x^2u_{tx}^2+\frac12\io u_{tx}^2 \gamma'(\Theta)\Theta_{xx}\nn\\
                            &\le& -\frac{c_\gamma}{2}\io u_{txx}^2     -\frac12\dt \io a(x,t) u_{xx}^2 +\frac{c_2}{2} \io u_{xx}^2+ \frac{c_2^2}{c_\gamma}\io u_x^2  +\frac{C_{f'}^2}{c_\gamma}\io \Theta_x^2\nn\\
                            & & +\frac{ c_1^2}{D}\io u_{tx}^4 + \frac D4\io     \Theta_{xx}^2 
\eea
for all $t\in(0,T_0)\cap(0,\tm)$ and further for $c_3=c_3(\gamma,D):=\frac{C_\gamma^2D+1}{D^2}$ and $c_4=c_4(f):= C_f^4$
\bea{Theta3}
\frac12\dt\io \Theta_x^2 +\frac D2\io \Theta_{xx}^2 \le  c_3\io u_{tx}^4+c_4\io( \Theta+1)^{4\alpha}  \eea
for all $t\in(0,T_0)\cap(0,\tm)$.
Since $u_t\in C^{\beta,\frac\beta2}(\bom\times[0,T])$ for every $T\in (0,T_0)\cap(0,\tm)$, we are able to apply Lemma \ref{A1} for each $t\in (0,T_0)\cap (0,\tm)$ on $\vp(x)\equiv u_t(x,t)$ and $\eta=\frac {c_\gamma D}{4(c_1^2+c_3D)} >0 $ to conclude for some $c_5=c_5(\gamma,D)>0$
\bea{utxHoelderAnw}\left(\frac{c_1^2+c_3 D}{D}\right)\io u_{tx}^4 \le \frac{c_\gamma}{4}   \io u_{txx}^2+ c_5\big\|u_t\big\|_{L^\infty(\Om)}^4\qquad\mbox{ for all }t\in(0,T_0)\cap(0,\tm).\eea
We note that since $\alpha\in(0,5/6)$, we are always able to choose $\alpha_0\in(5/6,1)$ such that $\alpha<\alpha_0$ holds and thus $\io (\Theta+1)^{4\alpha}\le |\Om|+ \io (\Theta+1)^{4\alpha_0}$. We then infer by another application of a Gagliardo-Nirenberg inequality that there exists some $c_6=c_6(\Om,f)>0$ such that
$$ c_4\big\|\vp\big\|_{L^{4\alpha_0}(\Om)}^{4\alpha_0} \le c_6\big\|\vp_x\big\|_{L^2(\Om)}^{4\alpha_0 \lambda}\big\|\vp\big\|_{L^1(\Om)}^{(1-\lambda)4\alpha_0} + c_6\big\| \vp \big\|_{L^1(\Om)}^{4\alpha_0},$$
for all $\vp\in C^1(\bom)$, with $\lambda=\frac{4\alpha_0-1}{6\alpha_0}$, which, when applied to $\vp=\Theta+1$, leads to
\bea{Theta4}c_4\io (\Theta+1)^{4\alpha_0}&=& c_4\big\|\Theta+1\big\|_{L^{4\alpha_0}(\Om)}^{4\alpha_0} \nn\\
&\le& c_6\big\|\Theta_x\big\|_{L^2(\Om)}^{4\alpha_0 \lambda}\big\|\Theta+1\big\|_{L^1(\Om)}^{(1-\lambda)4\alpha_0} + c_6\big\| \Theta+1 \big\|_{L^1(\Om)}^{4\alpha_0}
\eea
 for all $t\in(0,T_0)\cap(0,\tm)$. Since $\lambda\cdot2\alpha_0=\frac{4\alpha_0-1}{3}<1$ we are able to employ Young's inequality once more to obtain
\bea{Theta5} c_6\big\|\Theta_x\big\|_{L^2(\Om)}^{4\alpha_0 \lambda}\big\|\Theta+1\big\|_{L^1(\Om)}^{(1-\lambda)4\alpha_0} &\le& c_6 \io \Theta_x^2 +c_6\left(\io \Theta+1 \right)^{\frac{2\alpha_0+1}{2-2\alpha_0}}\nn\\
&\le& c_6 \io \Theta_x^2 + c_6 (C_1+|\Om|)^{12}\qquad \mbox{ for all }t\in(0,T), \eea 
where we have used the fact that $C_1>1$. Since in view of Lemma \ref{Mass} one can find $c_7=c_7(a,M,T_0)>0$ such that
\bea{Theta6}
c_6 \big\| \Theta+1\big\|_{L^1(\Om)}^{4\alpha_0} \le c_7
\eea
for all $t\in(0,T_0)\cap (0,\tm)$. We compute by combining  \eqref{c1Schranke}-\eqref{Theta6} with Lemma \ref{Mass} and Lemma \ref{MoserAnw}
\begin{align}\label{Yzus}
\dt\frac12\Bigg\{ \io u_{tx}^2&+\io \Theta_x^2  +\io a(x,t) u_{xx}^2\Bigg\}+ \frac{c_\gamma}{4} \io u_{txx}^2 \nn\\
&\le   \left( \frac{ c_1^2}D+c_3\right)\io u_{tx}^4 +\frac{C_{f'}^2}{c_\gamma}\io \Theta_x^2 
  +c_4\io (\Theta+1)^{4\alpha}+\frac{c_2} 2\io u_{xx}^2+\frac{c_2^2}{c_\gamma}\io u_x^2\nn\\
&\le \frac{c_\gamma}{4}   \io u_{txx}^2+ c_5\big\|u_t\big\|_{L^\infty(\Om)}^4+ c_8\io \Theta_x^2 +c_6(|\Om|+C_1)^{12}+c_7+\frac{c_2^2}{2}\io a(x,t) u_{xx}^2+\frac{c_2^2}{c_\gamma}C_1\nn\\
&\le \frac{c_\gamma}{4}   \io u_{txx}^2 + c_8\io \Theta_x^2  +c_8\io a(x,t)u_{xx}^2+c_9
\end{align}
for all $t\in(0,T_0)\cap(0,\tm)$, where $c_8:= \frac{C_{f'}^2}{c_\gamma}+c_6+\frac{c_2^2}{c_\gamma}$ and  $c_9:=c_5C_4^4+c_6(|\Om|+C_1)^{12}+\frac{c_2^2 C_1}{c_\gamma}+c_7$.
To prepare a comparison argument, we define 
$$y(t):=1+\frac12\io u_{tx}^2(\cdot,t)+\frac12 \io \Theta_x^2(\cdot,t)+\frac12\io a(\cdot,t) u_{xx}^2(\cdot,t) $$
for all $t\in (0,\tm)$ and note with respect to \eqref{großy1} that
$$y(0)=1+\frac{1}{2}\io u_{0tx}^2+ \frac{1}{2}\io \Theta_{0x}^2 +\frac{1}{2}\io a(\cdot,0) u_{0xx}^2 \le c_2\cdot M.$$ From \eqref{Yzus}, we may conclude that for $c_{10}:=\max\{2c_8,c_9\}$
$$y'(t)\le c_{10} y(t)\qquad \mbox{ for all }t\in(0,T_0)\cap(0,\tm)$$
and by a simple comparison argument 
$$y(t)\le  c_2M\cdot e^{c_{10}t} \qquad \mbox{ for all }t\in[0,T_0]\cap[0,\tm). $$
\qed
With a bound for $\io \Theta_x^2$ in hand, the proof of \ref{theoglobal} is complete.\\
\proofc of Theorem \ref{theoglobal}.\\
In view of \eqref{Ext}, the proof follows directly from Lemma \ref{großy}.
\qed

\textbf{Data availability statement.}
Data sharing is not appicable to this article since no datasets were generated or analyzed during the current study.

\textbf{Acknowledgment.}
The author acknowledges the support provided by the Deutsche Forschungsgemeinschaft
(Project No. 444955436) and further declares that he has no conflict of interest.

\printbibliography

\end{document}